\documentclass[12pt,twoside]{amsart}
\usepackage{amssymb,amsmath,amsthm, amscd, enumerate, mathrsfs, upgreek}
\usepackage{graphicx, hhline, tikz}
\usepackage[colorlinks=true,pagebackref,hyperindex]{hyperref}
\usepackage[all]{xy}
\usepackage{color}
\usepackage[backrefs, alphabetic, initials]{amsrefs}
\usepackage{color} 
\usepackage{latexsym}
\usepackage[T1]{fontenc} 
\usepackage{fancyhdr}
\usepackage{enumerate}
\usepackage{tikz-cd}
\usepackage{array, longtable}
\usepackage{caption}
\usepackage{mathtools}
\newcolumntype{C}{>{$}c<{$}}
\newcolumntype{L}{>{$}l<{$}}

\title[An effective upper bound for Fano indices]
{An effective upper bound for Fano indices of canonical Fano threefolds, I}

\date{\today, version 0.01}

\author{Chen Jiang}
\address{Chen Jiang, Shanghai Center for Mathematical Sciences \& School of Mathematical Sciences, Fudan University, Shanghai 200438, China}
\email{\href{chenjiang@fudan.edu.cn}{chenjiang@fudan.edu.cn}}
\urladdr{\href{https://chenjiangfudan.github.io/home/index.html}{https://chenjiangfudan.github.io/home/index.html}}

\author{Haidong Liu}
\address{Haidong Liu, Sun Yat-Sen University, School of Mathematics, Guangzhou 510275, China}
\email{\href{liuhd35@mail.sysu.edu.cn}{liuhd35@mail.sysu.edu.cn},\href{jiuguiaqi@gamil.com}{jiuguiaqi@gmail.com}}
\urladdr{\href{https://sites.google.com/view/liuhaidong}{https://sites.google.com/view/liuhaidong}}

\DeclareMathOperator{\rank}{rank}

\DeclareMathOperator{\mult}{mult}

\DeclareMathOperator{\reg}{reg}
\DeclareMathOperator{\Sing}{Sing}

\DeclareMathOperator{\Cl}{Cl}

\DeclareMathOperator{\red}{red}

\newcommand\lcm{{\text{l.c.m.}}}

\newcommand{\qQ}{\text{\rm q}_{\mathbb{Q}}}

\newtheorem{thm}{Theorem}[section]
\newtheorem{lem}[thm]{Lemma}
\newtheorem{prop}[thm]{Proposition}
\newtheorem{conj}[thm]{Conjecture}
\newtheorem{cor}[thm]{Corollary}

\theoremstyle{definition}
\newtheorem{ex}[thm]{Example}
\newtheorem{defn}[thm]{Definition}
\newtheorem{rem}[thm]{Remark}

\newtheorem{claim}[thm]{Claim}
\newtheorem{algo}[thm]{Algorithm}

\makeatletter
 
 \@addtoreset{equation}{section}
\makeatother

\begin{document}

\begin{abstract}
    Let $X$ be a $\mathbb Q$-factorial weak Fano $3$-fold with at worst isolated canonical singularities. We show that the $\mathbb Q$-Fano index of $X$ is at most $61$.
\end{abstract}

\maketitle 
\tableofcontents

\section{Introduction}\label{sec1}

Throughout this paper, we work over the complex number field $\mathbb C$ and adopt the standard notation from \cites{kollar-mori}.

A normal projective variety is called \emph{Fano} (resp., \emph{weak Fano}) if its anti-canonical divisor is ample (resp., nef and big). According to the minimal model program,  Fano varieties with mild singularities serve as fundamental building blocks of algebraic varieties.

For a Fano variety  $X$ with mild singularities (e.g., terminal or canonical singularities), its $\mathbb Q$-Fano index $\qQ(X)$ is defined to be 
\begin{align*} 
    \qQ(X) {}&\coloneq \max\{q \mid -K_X \sim_{\mathbb Q}qA, \quad A\in \Cl (X) \}. 
\end{align*}
The $\mathbb{Q}$-Fano index of a terminal Fano $3$-fold is quite well-understood, which belongs to $\{1,\dots,9, 11, 13,17, 19\}$ (see \cites{suzuki, prokhorov2010}) and it plays a crucial role in the classification of terminal Fano $3$-folds \cites{suzuki,bs1,bs2,prokhorov2007,prokhorov2010,prokhorov2013,kasprzyk, liu-liu, liu-liu2024}. However, for canonical Fano $3$-folds, progress has been limited due to the absence of Riemann--Roch formula and the classification of canonical singularities, and even a good upper bound for $\mathbb{Q}$-Fano indices remains unknown. In literature, for a canonical Fano $3$-fold $X$, we only have a very rough upper bound $\qQ(X)\leq 840^2\cdot 324\approx 2.3\cdot 10^8$ (cf. \cite{Birkar4fold}*{Lemma~2.3}): suppose that $-K_X \sim_{\mathbb Q}qA$ for some ample Weil divisor $A$, then we know that $rK_X$ is Cartier for some positive integer $r\leq 840$ by \cite{CJ2016}*{Proposition~2.4} and $(-K_X)^3\leq 324$ by \cite{JZ}, which implies that 
\[
    q\leq qA\cdot (-rK_X)^2= r^2(-K_X)^3\leq 840^2\cdot 324.
\]

Motivated by the terminal case, it should be expected that weighted projective spaces have larger $\mathbb{Q}$-Fano indices, so we are interested in the following conjecture:
\begin{conj}[{cf. \cite{wang}*{Conjecture~3.7}}]\label{conj.fanobound}
    Let $X$ be a canonical weak Fano $3$-fold. Then $\qQ(X)\leq 66$.
\end{conj}

Kasprzyk classified terminal or canonical toric Fano $3$-folds in \cites{kasprzyk, kasprzyk2}. The following are some motivating examples with large $\mathbb{Q}$-Fano indices.

\begin{ex}[\cite{kasprzyk2}]
    \begin{enumerate}
        \item The weighted projective space $\mathbb P(5,6,22,33)$ is a $\mathbb Q$-factorial canonical Fano $3$-fold of Picard number $1$. It has non-isolated singularities and its $\mathbb Q$-Fano index is $66$. 
        
        \item The weighted projective space $\mathbb P(3,5,11,19)$ is a $\mathbb Q$-factorial Fano $3$-fold of Picard number $1$ with isolated canonical singularities. Its $\mathbb Q$-Fano index is $38$. 
        
        \item The weighted projective space $\mathbb P(5,8,9,11)$ is a $\mathbb Q$-factorial Fano $3$-fold of Picard number $1$ with isolated canonical singularities which is non-Gorenstein at a crepant center. Its $\mathbb Q$-Fano index is $33$. 
    \end{enumerate}
\end{ex}

In this paper, we provide an affirmative answer to Conjecture~\ref{conj.fanobound} for $\mathbb Q$-factorial weak Fano $3$-folds with isolated canonical singularities. 

\begin{thm}\label{thm.main}
    Let $X$ be a $\mathbb Q$-factorial weak Fano $3$-fold with at worst isolated canonical singularities. Then $ \qQ(X)\leq 61$.
\end{thm}

The key idea is to use a Kawamata--Miyaoka type inequality established in \cite{jiang-liu-liu} relating $c_1(X)^3$ with $c_2(X)\cdot c_1(X)$, along with a characterization of $c_1(X)^3$ in terms of $\mathbb Q$-Fano indices. This approach brings us very close to proving Conjecture~\ref{conj.fanobound} for isolated singularities, except for two particularly stubborn cases with $\qQ(X)\in \{67, 71\}$. We put our main effort into addressing these two exceptional cases by studying the geometry of $X$ in more detail. To this end, we develop a Riemann--Roch type  formula for canonical $3$-folds (see \S\,\ref{sec.rr}) and study the geometry of foliations of rank $2$ (see \S\,\ref{sec.foliation}).

\section{Preliminaries}

 
\subsection{Singularities}

Let $X$ be a normal variety such that $K_X$ is $\mathbb Q$-Cartier. The \emph{Gorenstein index} $r_X$ of $X$ is defined as the smallest positive integer $m$ such that $mK_X$ is Cartier. Let $f\colon Y\to X$ be a proper birational morphism. A prime divisor $E$ on $Y$ is called \emph{a divisor over} $X$, and the image $f(E)$ is called the \emph{center} of $E$ on $X$. 
Write
\[
    K_Y=f^*K_X+\sum_Ea(E,X)E,
\]
where $a(E,X)\in \mathbb Q$ is called the \emph{discrepancy} of $E$. We say that $X$ has \emph{terminal} (resp., \emph{canonical}) singularities if $a(E,X)>0$ (resp., $a(E,X)\geq 0$) for any exceptional divisor $E$ over $X$. Often we just simply say that $X$ is {\it terminal} or {\it canonical}, respectively. A \emph{crepant center} of $X$ is the center of an exceptional divisor $E$ over $X$ with $a(E,X)=0$. We say that $X$ has at worst {\it isolated singularities} if its singular locus $\Sing(X)$ is either empty or consists of closed points. For example, it is well-known that normal surface singularities and $3$-dimensional terminal singularities are isolated (\cite{kollar-mori}*{Corollary~5.18}).

If $X$ is a canonical variety, then by the minimal model program, there exists a \emph{terminalization} $f\colon Y\to X$ such that $Y$ is a terminal variety with $f^*K_X=K_Y$ and $f$ is projective birational (see \cite{BCHM}*{Corollary~1.4.3}).

\subsection{Fano index}\label{subsec.fi}

Let $X$ be a canonical weak Fano variety. We can define the \emph{$\mathbb Q$-Fano index} of $X$ by 
\begin{align*} 
    \qQ(X) {}&\coloneq \max\{q \mid -K_X \sim_{\mathbb Q}qA, \quad A\in \Cl (X) \}. 
\end{align*}
It is known that $\Cl (X)$ is a finitely generated Abelian group, so $\qQ(X)$ is a positive integer. For more details, see \cite{ip}*{\S\,2} or \cite{prokhorov2010}.

We can use the following lemma to reduce Theorem~\ref{thm.main} to the Picard number $1$ case.
 
\begin{lem}[{cf. \cite{jiang-liu-liu}*{Lemma~5.1, Proposition~5.2}}]\label{lem.reducetopicard1}
    Let $X$ be a $\mathbb{Q}$-factorial weak Fano $3$-fold with at worst isolated canonical singularities and with $\qQ(X)\geq 7$. Then there exists a $3$-fold $Y$ with the following properties:
    \begin{enumerate}
        \item $Y$ is a $\mathbb Q$-factorial Fano $3$-fold of Picard number $1$ with at worst isolated canonical singularities;
        
        \item $\qQ(Y)\geq \qQ(X)\geq 7$.     
    \end{enumerate}
\end{lem}
\begin{proof}
    Fix a Weil divisor $A$ such that $-K_X \sim_{\mathbb Q}\qQ(X)A$. We can run a $K$-MMP on $X$ which ends up with a Mori fiber space $Y\to T$ where $Y$ is $\mathbb{Q}$-factorial and canonical. Then $-K_{Y}\sim_{\mathbb{Q}}\qQ(X) A'$ where $A'$ is the strict transform of $A$ on $Y$.
    
    If $\dim T\geq 1$, then for a general fiber $F$ of $Y\to T$, we have $-K_F\sim_{\mathbb{Q}}\qQ(X) A'|_F$ and $F$ is either $\mathbb{P}^1$ or a canonical del Pezzo surface, but this contradicts \cite{wang}*{Proposition~3.3}  as $\qQ(X)\geq 7$.
    
    Hence $\dim T=0$, which means that $Y$ is a $\mathbb Q$-factorial canonical Fano $3$-fold of Picard number $1$. Suppose that $Y$ has worse than isolated singularities, then $\Sing(Y)$ contains a curve $C$, which means that $C$ is a crepant center as $3$-dimensional terminal singularities are isolated. As $X$ has canonical singularities and $X\dashrightarrow Y$ is obtained by a $K$-MMP, $X$ and $Y$ are isomorphic over the generic point of $C$ by \cite{kollar-mori}*{Lemma~3.38}, which contradicts the fact that $X$ has only isolated singularities. 
\end{proof}

\subsection{Reid's basket and formula}

Let $X$ be a canonical projective $3$-fold. According to Reid \cite{reid}*{(10.2)}, there is a collection of pairs of integers (permitting weights)
\[
    B_{X}=\{(r_{i}, b_{i}) \mid i=1, \cdots, s ; 0<b_{i}\leq\frac{r_{i}}{2} ; b_{i} \text{ is coprime to } r_{i}\}
\]
associated to $X$, called {\it Reid's basket}, where a pair $(r, b)$ corresponds to an orbifold point of type $\frac{1}{r}(1, -1, b)$ which comes from locally deforming non-Gorenstein singularities of a terminalization of $X$. In other words, $B_X$ is identified with the collection of virtual orbifold points of $X$. By definition, if $Y$ is a terminalization of $X$, then $B_Y=B_X$. Denote by $\mathcal{R}_X$ the collection of $r_i$ (permitting weights) appearing in $B_X$. Note that the Gorenstein index $r_X$ of $X$ is just $\lcm\{r\mid r\in \mathcal{R}_X\}$.

Let $D$ be a $\mathbb{Q}$-Cartier Weil divisor on $X$. We say that $D$ satisfies {\it Reid's condition} if locally at any point $P\in X$, $D\sim iK_X$ for some integer $i$ (depending on $P$). Note that if $X$ has terminal singularities, then $D$ always satisfies Reid's condition by \cite{kawamata-crepant}*{Corollary~5.2}. For $D$ satisfying Reid's condition, according to \cite{reid}*{(10.2)}, there exists an orbifold Riemann--Roch formula, called \emph{Reid's formula}:
\begin{equation}\label{eq.reidformula}
    \chi(X, \mathcal O_X(D))=\chi(X, \mathcal O_X)+\frac{1}{12}D\cdot (D-K_X)\cdot (2D-K_X)+\frac{1}{12}c_2(X)\cdot D+\sum_{Q\in B_X}c_Q(D),
\end{equation}
where the last sum runs over Reid’s basket of orbifold points. If the orbifold point $Q$ corresponds to $(r,b)\in B_X$ and $i= i_D$ is the local index of divisor $D$ at $Q$ (i.e., after taking terminalization and local deformation, $D\sim iK_X$ around $Q$ for some integer $i\geq 0$), then 
\[
    c_Q(D)\coloneq -\frac{i(r^2-1)}{12r}+\sum_{j=0}^{i-1}\frac{\overline{jb}(r-\overline{jb})}{2r}.
\]
Here the symbol $\overline{\cdot}$ means the smallest residue mod $r$ and $\sum_{j=0}^{-1}\coloneq 0$. The term $c_2(X)\cdot D$ is defined to be
$c_2(W)\cdot \pi^*D$ for any resolution $\pi\colon W\to X$. In particular, by \cite{reid}*{(10.3)}, we have
\begin{align}\label{eq.range}
    c_2(X)\cdot c_1(X) + \sum_{r\in \mathcal{R}_X} \left(r-\frac{1}{r}\right)=24\chi(X, \mathcal O_X),
\end{align}
and 
\begin{align}\label{eq.RR-Fano}
    \frac{1}{2}c_1(X)^3+3\chi(X, \mathcal O_X)- \sum_{(r,b)\in B_X}\frac{b(r-b)}{2r}= \chi(X,\mathcal{O}_X(-K_X))\in \mathbb Z_{\geq 0}.
\end{align}

\begin{rem}\label{rem.posofc2}
    For a canonical weak Fano $3$-fold $X$, $\chi(X, \mathcal{O}_X)=1$ by the Kawamata--Viehweg vanishing theorem and $c_2(X)\cdot c_1(X)>0$ by \cite{ijl}*{Corollary 7.3}. 
\end{rem}

\subsection{Kawamata--Miyaoka type inequality for canonical Fano threefolds}\label{sub.KMineq}

The following theorem is a detailed version of the Kawamata--Miyaoka type inequality of \cite{jiang-liu-liu}*{Theorem~3.8}. Here we refer to \cite{jiang-liu-liu}*{\S\,3} for the definition of \emph{generalized second Chern class} $\hat{c}_2(X)$ of a $3$-fold $X$. For the application in this paper, we just note that if $X$ has only isolated singularities, then $\hat{c}_2(X)$ is just the usual $c_2(X)$.

\begin{thm}[\cite{jiang-liu-liu}*{Theorem~3.8}]\label{thm.kmineqforisolated}
    Let $X$ be a $\mathbb{Q}$-factorial canonical Fano $3$-fold of Picard number $1$. Let $q\coloneq \qQ(X)$ be the $\mathbb{Q}$-Fano index of $X$. Let
    \begin{equation}\label{eq.hnfiltration}
        0=\mathcal{E}_0\subsetneq \mathcal{E}_1\subsetneq \dots \subsetneq \mathcal{E}_l=\mathcal{T}_X
    \end{equation}
    be the Harder--Narasimhan filtration of the tangent sheaf $\mathcal T_X$ with respect to $c_1(X)$, where $1\leq l\leq 3$. Denote by $r_1\coloneq\rank \mathcal{E}_1$ and take $p$ to be the integer such that $c_1(\mathcal{E}_{l-1})\equiv \frac{p}{q}c_1(X)$. Here $p<q$, and $p>\frac{l-1}{l}q$ if $l>1$. Then we have
    \[
        c_1(X)^3 \leq 
        \begin{dcases}
        3 \hat{c}_2(X)\cdot c_1(X) & \text{if } (l, r_1)=(1,3);\\
        \frac{16}{5} \hat{c}_2(X)\cdot c_1(X) & \text{if } (l, r_1)=(2,1);\\
        \frac{4q^2}{p(4q-3p)} \hat{c}_2(X)\cdot c_1(X)\leq \frac{4q^2}{q^2+2q-3} \hat{c}_2(X)\cdot c_1(X)& \text{if } (l, r_1)=(2,2);\\
        \frac{4q^2}{-4p^2+6pq-q^2} \hat{c}_2(X)\cdot c_1(X)\leq \frac{4q^2}{q^2+2q-4} \hat{c}_2(X)\cdot c_1(X)& \text{if } (l, r_1)=(3,1).
        \end{dcases}
    \]
    In particular, we have $c_1(X)^3<4\hat{c}_2(X)\cdot c_1(X)$.
\end{thm}

\begin{proof}
    We refer the reader to \cite{jiang-liu-liu}*{Proof of Theorem~3.8} with some necessary explanations. Recall that in \cite{jiang-liu-liu}*{Proof of Theorem~3.8}, we denote $\mathcal{F}_i\coloneqq (\mathcal{E}_i/\mathcal{E}_{i-1})^{**}$ and denote $q_i\geq 1$ the unique positive integer such that $c_1(\mathcal{F}_i)\equiv \frac{q_i}{q}c_1(X)$. Then $p=\sum_{i=1}^{l-1}q_i$. 
    
    If $(l, r_1)=(1,3)$, that is, $\mathcal{T}_X$ is semistable with respect to $c_1(X)$, then the inequality follows directly from the $\mathbb{Q}$-Bogomolov--Gieseker inequality \cite{KMM94}*{Lemma~6.5}.
    
    If $(l, r_1)=(2,1)$, then the inequality follows from \cite{jiang-liu-liu}*{Proof of Theorem~3.8, Case~1}.
    
    If $(l, r_1)=(2,2)$, then from \cite{jiang-liu-liu}*{Proof of Theorem~3.8, Case~2}, we have $p=q_1$ which means that $\frac23 q< p\leq q-1$ and
    \[
        6 \hat{c}_2(X)\cdot c_1(X) - 2c_1(X)^3 \geq -\frac{(3p-2q)^2}{2q^2} c_1(X)^3,
    \]
    which yields the desired inequality.
    
    If $(l, r_1)=(3,1)$, then from \cite{jiang-liu-liu}*{Proof of Theorem~3.8, Case~3}, we have $p=q_1+q_2$ with $2 \leq q_2 \leq q_1- 1 \leq \frac q2- 1$ and
    \begin{align*}
        6\hat{c}_2(X)\cdot c_1(X) - 2 c_1(X)^3 
        & \geq - \left((q_1-q_2)^2 + (2q_1+q_2-q)^2 + (q_1+2q_2-q)^2\right) \cdot \frac{1}{q^2} c_1(X)^3 \\
        & = - \left((2q_1-p)^2 + (q_1+p-q)^2 + (2p-q_1-q)^2\right) \cdot \frac{1}{q^2} c_1(X)^3 \\
        & \geq - \left(6 p^2 - 9pq + \frac{7}{2}q^2\right) \frac{1}{q^2} c_1(X)^3,
    \end{align*}
    which yields the desired inequality. Here for the last inequality, we use the fact that the first two leading terms of $ (2q_1-p)^2 + (q_1+p-q)^2 + (2p-q_1-q)^2$ as a polynomial of $q_1$ is $6q_1^2-6pq_1$, so it attains the maximal value at $q_1=\frac{q}{2}$ as $q_1>\frac{p}{2}$. 
\end{proof}

\section{Searching for the upper bound of Fano indices} 

Firstly, we recall a special case of \cite{jiang-liu-liu}*{Theorem 4.2} for Fano $3$-folds with at worst isolated canonical singularities.

\begin{thm}[\cite{jiang-liu-liu}*{Theorem 4.2}]\label{thm.degreeandindex}
    Let $X$ be a Fano $3$-fold with at worst isolated canonical singularities. Then $r_Xc_1(X)^3/\qQ(X)^2$ is a positive integer, where $\qQ(X)$ is the $\mathbb Q$-Fano index of $X$ and $r_X$ is the Gorenstein index of $X$.
\end{thm}

\begin{proof}
    It follows directly from \cite{jiang-liu-liu}*{Theorem 4.2 (1)} once we notice that the number $J_A=1$ in \cite{jiang-liu-liu}*{Theorem 4.2} as $X$ is smooth in codimension 2.
\end{proof}

Then, we can search for the upper bound of $\mathbb Q$-Fano indices for $\mathbb Q$-factorial Fano $3$-folds of Picard number $1$ with at worst isolated canonical singularities by combining Reid's Riemman--Roch formula and the Kawamata--Miyaoka type inequality.

\begin{thm}\label{thm.isolatecase}
    Let $X$ be a $\mathbb Q$-factorial Fano $3$-fold of Picard number $1$ with at worst isolated canonical singularities. Consider the pair of integers $(l, r_1)$ in Theorem~\ref{thm.kmineqforisolated}. 
    \begin{enumerate}
        \item If $(l, r_1)=(1,3)$ or $(2,1)$, then $\qQ(X)\leq 61$.
        \item If $(l, r_1)=(2,2)$ or $(3,1)$, then $\qQ(X)\leq 71$.
        \item If $X$ is non-Gorenstein at some crepant center, then $\qQ(X)\leq 45$. 
    \end{enumerate}
    Moreover, if $\qQ(X)> 61$, then $\qQ(X)\in\{67, 71\}$ and the numerical data of $X$ is listed in Table~\ref{tab1}.
    {
    \begin{longtable}{LLLLL}
        \caption{Large Fano indices}\label{tab1}\\
        \hline
        B_X & r_X & r_Xc_1^3 & r_Xc_2c_1 & \qQ \\
        \hline
        \endfirsthead
        \multicolumn{4}{l}{{ {\bf \tablename\ \thetable{}} \textrm{-- continued}}}
        \\
        \hline 
        B_X & r_X & r_Xc_1^3 & r_Xc_2c_1 & \qQ \\
        \hline 
        \endhead
        \hline
        \hline \multicolumn{4}{c}{{\textrm{Continued on next page}}} \\ \hline
        \endfoot
        
        \hline \hline
        \endlastfoot
        \{(2,1),(3,1),(5,2), (11,1)\} & 330 & 3721 & 1361 & 61 \\
        \{(2,1),(3,1),(5,1), (11,2)\} & 330 & 4489 & 1361 & 67 \\
        \{(2,1),(3,1),(5,2), (11,1)\}& 330 & 5041 & 1361 & 71 \\
        \{(2,1),(3,1),(5,1), (11,3)\} & 330 & 5329 & 1361 & 73\\
    \end{longtable}
    } 
\end{thm}

\begin{proof}
    We employ a computer program to search for the largest $\qQ(X)$, by the algorithm outlined below.
    
    \begin{algo}\label{algo1}
        {\bf Step 1}. List all possible $(\mathcal R_X, c_2(X)\cdot c_1(X))$ satisfying \eqref{eq.range}. By Remark~\ref{rem.posofc2}, there are only finitely many candidates.
        
        {\bf Step 2}. Among the list in Step 1, list all possible 
        $(\mathcal R_X, c_2(X)\cdot c_1(X), c_1(X)^3, \qQ(X))$ satisfying
        \begin{equation}\label{eq.test}
        61^2\leq \qQ(X)^2\leq r_Xc_1(X)^3\leq br_Xc_2(X)\cdot c_1(X),
        \end{equation}
        where $b\in \{3, 3.2, 4\}$ and that $r_Xc_1(X)^3/\qQ(X)^2$ is a positive integer, according to Theorem~\ref{thm.kmineqforisolated} and Theorem~\ref{thm.degreeandindex}.
        
        {\bf Step 3}. Among the list in Step 2, 
        list all possible $B_X$ satisfying \eqref{eq.RR-Fano}. 
        
        The output is a list of numerical data $(B_X, c_2(X)\cdot c_1(X), c_1(X)^3, \qQ(X))$.
    \end{algo}
 
    If $(l, r_1)=(1,3)$ or $(2,1)$, then we set $b=3$ or $3.2$ respectively in Algorithm~\ref{algo1} and the only output in both cases is with $B_X=\{(2,1),(3,1),(5,2), (11,1)\}$ and $\qQ(X)=61$. 

    If $(l, r_1)=(2,2)$ or $(3,1)$, then we set $b=4$ in \eqref{eq.test} and the output is listed in Table~\ref{tab1}. The last case $\qQ(X)=73$ in Table~\ref{tab1} can be ruled out by Theorem~\ref{thm.kmineqforisolated}, as in this case we have
    \[
        \frac{5329}{1361}=\frac{c_1(X)^3}{c_2(X)\cdot c_1(X)}>\frac{4\cdot 73^2}{ 73^2+2\cdot 73-4}.
    \]

    If $X$ is non-Gorenstein at some crepant center, then by \cite{kawakita}*{Table 2}, $\mathcal R_X$ must contain one of the following: $\{2,2,2,2\}$, $\{3,3,3\}$, $\{2,4,4\}$, $\{5,5\}$, or $\{2,3,6\}$. We can put this additional restriction in Step 1 of Algorithm~\ref{algo1}, and set $b=4$ and $\qQ(X)\geq 33$ in Step 2 of Algorithm~\ref{algo1}. The output list shows that $\qQ(X)\leq 45$ and the equality holds only if $B_X=\{(4, 1), (5, 1), (5, 2), (7, 3)\}$. 
\end{proof}

The sequel of this paper is dedicated to ruling out the remaining two cases.

\section{A Riemann--Roch formula for canonical Fano threefolds}\label{sec.rr}

In the study of Fano indices on canonical Fano $3$-folds, one issue is that we could not apply Reid's Riemann--Roch formula as in general a Weil divisor does not satisfy Reid's condition. In this section, we establish a Riemann--Roch formula for canonical Fano $3$-folds based on Reid's formula. The idea is to take a special transform called a Weil pullback of the divisor on $X$ to a terminalization. 

\begin{defn}[Sequential terminalization]\label{def seq terminalization}
    Let $X$ be a variety with canonical singularities. A projective birational morphism $f\colon Y\to X$ together with a sequence of projective birational morphisms $f_k\colon X_k\to X_{k-1}$ for $1\leq k\leq m$ is called a {\it sequential terminalization} if the following properties are satisfied: 
    \begin{enumerate}
        \item $Y$ is terminal;
        \item $X_0=X$, $X_m=Y$, $f=f_m\circ \dots \circ f_1$; 
        \item for any $1\leq k\leq m$, $f_k$ is crepant, i.e., $f_k^*K_{X_{k-1}}=K_{X_k}$;
        \item for any $1\leq k\leq m$, denote by $E_k$ the sum of exceptional divisors of $f_k$, then either $E_k=0$, or $E_k$ is prime and $-E_k$ is $f_k$-nef.
    \end{enumerate} 
    For simplicity, we just say that $f\colon Y\to X$ is a sequential terminalization. 
\end{defn}

\begin{rem}
    We can always construct a sequential terminalization by applying \cite{BCHM}*{Corollary~1.4.3} repeatedly. See also \cite{kawakita}*{Corollary~2.6}.
\end{rem}

\begin{defn}[Weil pullback]
    Let $X$ be a variety with canonical singularities. Let $f\colon Y\to X$ be a sequential terminalization. For any $\mathbb{Q}$-Cartier Weil divisor $D$ on $X$, we define the {\it Weil pullback} $f^{\lfloor *\rfloor}(D)$ of $D$ on $Y$ as the following: keep the notation in Definition~\ref{def seq terminalization}, take $D_0=D$, and define inductively $D_k=\lfloor f_k^*D_{k-1} \rfloor$ for $1\leq k\leq m$, and finally set $f^{\lfloor *\rfloor}(D)\coloneq D_m$. Here we can inductively show that $D_k$ is $\mathbb{Q}$-Cartier since $f_k^*D_{k-1}$ and $E_k$ are $\mathbb{Q}$-Cartier. So $f^{\lfloor *\rfloor}(D)$ is a $\mathbb{Q}$-Cartier Weil divisor on $Y$. 
\end{defn}

We can compare the cohomologies of Weil pullback with the original divisor by the following lemma.

\begin{lem}\label{lem same coh}
    Let $X$ be a projective variety with canonical singularities. Let $f\colon Y\to X$ be a sequential terminalization. Then for any $\mathbb{Q}$-Cartier Weil divisor $D$ on $X$ and any $i\geq 0$, we have 
    \[
        H^i(Y, \mathcal O_Y(f^{\lfloor *\rfloor}(D)))=H^i(X, \mathcal O_X(D)).
    \]
\end{lem}

\begin{proof}
    Keep the notation in Definition~\ref{def seq terminalization}. It suffices to prove that \[H^i(X_k, \mathcal O_{X_k}(D_k))=H^i(X_{k-1}, \mathcal O_{X_{k-1}}(D_{k-1}))\] for every $k\geq 1$ and $i\geq 0$. 

    By construction, we have $D_k=\lfloor f_k^*D_{k-1} \rfloor=f_k^*D_{k-1}-a_kE_k$ for some $0\leq a_k<1$. So $f_{k*}\mathcal O_{X_k}(D_k)=\mathcal O_{X_{k-1}}(D_{k-1})$ by \cite{nakayama-zariski}*{Lemma~II.2.11}. Moreover, $D_k$ is $f_k$-nef and $f_k$-big and $K_{X_k}$ is $f_k$-trivial as $f_k$ is crepant. Hence the Kawamata--Viehweg vanishing theorem yields that
    \[
        R^jf_{k*}\mathcal O_{X_k}(D_k)=R^jf_{k*}\mathcal O_{X_k}(K_{X_k}+D_k-K_{X_k})=0
    \]
    for $j>0$. By a spectral sequence argument, it follows that 
    \[
        H^i(X_k, \mathcal O_{X_k}(D_k))=H^i(X_{k-1}, f_{k*}\mathcal O_{X_k}(D_k))=H^i(X_{k-1}, \mathcal O_{X_{k-1}}(D_{k-1}))
    \]
    for every $i\geq 0$. 
\end{proof}

The Weil pullback behaves not so well as it is not additive but only superadditive, namely, for $2$ divisors $D$ and $D'$, $f^{\lfloor *\rfloor}(D+D') \geq  f^{\lfloor *\rfloor}(D)+f^{\lfloor *\rfloor}(D')$. We have additivity in the following case.

\begin{lem}\label{lem pullback of D+G}
    Let $X$ be a variety with canonical singularities. Let $f\colon Y\to X$ be a sequential terminalization. Let $D$ and $D'$ be $\mathbb{Q}$-Cartier Weil divisors on $X$. If the usual pullback $f^*D'$ is a Weil divisor on $Y$, then $f^{\lfloor *\rfloor}(D+D') = f^{\lfloor *\rfloor}(D)+f^*D'$.
\end{lem}

\begin{proof}
    Keep the notation in Definition~\ref{def seq terminalization}. Take $D'_0=D'$, and define inductively $D'_k= f_k^*D'_{k-1} $ for $1\leq k\leq m$. Then as $f^*D'$ is a Weil divisor, $D'_k$ is a Weil divisor for $1\leq k\leq m$. This implies that 
    \[
        \lfloor f_k^*(D_{k-1}+D'_{k-1}) \rfloor=\lfloor f_k^*D_{k-1} +f_k^*D'_{k-1}\rfloor=\lfloor f_k^*D_{k-1} \rfloor +f_k^*D'_{k-1}=D_{k}+D'_k 
    \] 
    for $1\leq k\leq m$. So we get $f^{\lfloor *\rfloor}(D+D') = f^{\lfloor *\rfloor}(D)+f^*D'$.
\end{proof}
 
\begin{prop}\label{prop.rr.new}
    Let $X$ be a projective canonical $3$-fold. Let $f\colon Y\to X$ be a sequential terminalization. Let $D$ be a $\mathbb{Q}$-Cartier Weil divisor on $X$. Then
    \begin{align*}
        {}&\chi(X, \mathcal{O}_X(D))-\chi(X, \mathcal{O}_X(D+K_X))\\
        ={}&-\frac{1}{2}(f^{\lfloor *\rfloor}(D))^2\cdot K_Y+2\chi(X, \mathcal{O}_X)-\sum_{(r, b)\in B_X}\frac{\overline{ i_{f^{\lfloor *\rfloor}(D)} b}(r-\overline{i_{f^{\lfloor *\rfloor}(D)}b})}{2r},
    \end{align*}
    where $i_{f^{\lfloor *\rfloor}(D)} $ is the local index of ${f^{\lfloor *\rfloor}(D)}$ at the orbifold point of type $(r,b)\in B_X$.
\end{prop}

\begin{proof}
    We simply denote $S\coloneq f^{\lfloor *\rfloor}(D)$. As $f$ is crepant, by Lemma~\ref{lem pullback of D+G}, 
    \[
        f^{\lfloor *\rfloor}(D+K_X)={S}+K_Y.
    \] 
    So by Lemma~\ref{lem same coh} and Reid's formula \eqref{eq.reidformula},
    \begin{align*}
        {}&\chi(X, \mathcal{O}_X(D))-\chi(X, \mathcal{O}_X(D+K_X))\\
        ={}&\chi(Y, \mathcal{O}_Y(S))-\chi(Y, \mathcal{O}_Y(S+K_Y))\\
        ={}&-\frac{1}{2}S^2\cdot K_Y-\frac{1}{12}c_2(Y)\cdot K_Y+\sum_{Q\in B_Y}(c_Q(S)-c_Q(S+K_Y)).
    \end{align*} 
    For an orbifold point $Q$ corresponding to $(r, b)\in B_Y$, if $i$ is the local index of $S$ at $Q$, then $i+1$ is the local index of $S+K_Y$ at $Q$ and hence
    \[
        c_Q(S)-c_Q(S+K_Y)= \frac{r^2-1}{12r}-\frac{\overline{ib}(r-\overline{ib})}{2r}.
    \]
    So the conclusion follows from \eqref{eq.range} and the fact that $B_X=B_Y$.
\end{proof}

\begin{cor}\label{cor.rr.sA}
    Let $X$ be a canonical Fano $3$-fold. Suppose that $-K_X\equiv qA$ for some positive rational number $q$ and ample Weil divisor $A$. Let $f\colon Y\to X$ be a sequential terminalization. Then for any integer $s$ with $0<s<q$,
    \[
        h^0(X,\mathcal O_X(sA))= 
        -\frac{1}{2}(f^{\lfloor *\rfloor}(sA))^2\cdot K_Y+2-\sum_{(r,b)\in B_X}\frac{\overline{i_sb}(r-\overline{i_sb})}{2r}
    \]
    where $i_s=i_{f^{\lfloor *\rfloor}(sA)} $ is the local index of ${f^{\lfloor *\rfloor}(sA)}$ at the orbifold point of type $(r,b)\in B_X$.
\end{cor}

\begin{proof}
    Since $-K_X$ is ample, all higher cohomologies of $sA$ and $sA+K_X$ vanish 
    by the Kawamata--Viehweg vanishing theorem. Also $h^0(X, \mathcal{O}_X(sA+K_X))=0$ as $sA+K_X\equiv (s-q)A$ and $s<q$. So we have 
    \begin{align*}
        \chi(X, \mathcal{O}_X){}&=h^0(X, \mathcal{O}_X)=1,\\
        \chi(X, \mathcal{O}_X(sA)){}&=h^0(X, \mathcal{O}_X(sA)),\\
        \chi(X, \mathcal{O}_X(sA+K_X)){}&=0.
    \end{align*}
    The the conclusion follows directly from Proposition~\ref{prop.rr.new}.
\end{proof}

\begin{lem}[{cf. \cite{prokhorov2010}*{Proposition~2.9}}]\label{lem.torsionRR}
    Let $X$ be a canonical Fano $3$-fold. Let $f\colon Y\to X$ be a sequential terminalization. Let $D$ be a $\mathbb{Q}$-Cartier Weil divisor on $X$ such that $D\equiv 0$ but $D\not\sim 0$. Then
    \begin{align*}
        2=\sum_{(r,b)\in B_X}\frac{\overline{ i_{f^{\lfloor *\rfloor}(D)} b}(r-\overline{i_{f^{\lfloor *\rfloor}(D)}b})}{2r}+\frac{1}{2}(f^{\lfloor *\rfloor}(D))^2\cdot K_Y,
    \end{align*}
    where $i_{f^{\lfloor *\rfloor}(D)} $ is the local index of ${f^{\lfloor *\rfloor}(D)}$ at the orbifold point of type $(r,b)\in B_X$.
\end{lem}

\begin{proof}
    Since $-K_X$ is ample, all higher cohomologies of $ D$ and $-D$ vanish 
    by the Kawamata--Viehweg vanishing theorem. Moreover, 
    \[
        h^0(X, \mathcal{O}_X(D))=h^0(X, \mathcal{O}_X(-D))=0
    \]
    as $D\equiv 0$ but $D\not\sim 0$. So the conclusion follows from Proposition~\ref{prop.rr.new} and the Serre duality. 
\end{proof}

\begin{rem}
    In general, there are two main issues in applications of the formula in Corollary~\ref{cor.rr.sA}: the first is to determine the intersection number $(f^{\lfloor *\rfloor}(sA))^2\cdot K_Y$, and the second is to determine the local index $i_t$. Both issues are due to the complicated definition of $f^{\lfloor *\rfloor}(sA)$. We use the following two lemmas to treat these issues in some special case.
\end{rem}

\begin{lem}\label{lem.DH}
    Let $X$ be a projective variety of dimension $n$ with canonical singularities. Let $f\colon Y\to X$ be a sequential terminalization. For any $\mathbb{Q}$-Cartier Weil divisor $D$ on $X$ and any nef divisor $H$ on $X$, $(f^{\lfloor *\rfloor}(D))^2\cdot (f^*H)^{n-2}\leq D^2\cdot H^{n-2}$. Moreover, the equality holds if $X$ is smooth in codimension $2$. 
\end{lem}

\begin{proof}
    As nef divisors are limits of ample divisors, it suffices to prove the statement for ample $H$. After taking a multiple, we may assume that $H$ is very ample. We may take very general $H_1, \cdots, H_{n-2}\in |H|$ such that $S\coloneq H_1\cap\cdots \cap H_{n-2}\subset X$ and $T\coloneq f^*H_1\cap\cdots \cap f^*H_{n-2}\subset Y$ are normal projective surfaces. 
    
    By construction, we may write $f^{\lfloor *\rfloor}(D)=f^*D-E$ where $E$ is an effective $f$-exceptional $\mathbb{Q}$-divisor. Here $E|_T$ is exceptional over $S$. So by the projection formula and the Hodge index theorem, 
    \[
        (f^{\lfloor *\rfloor}(D))^2\cdot (f^*H)^{n-2}=(f^*D|_T-E|_T)^2=f^*D|_T^2+E|_T^2\leq D|_S^2=D^2\cdot H^{n-2}.
    \]
    
    Since $f$ is crepant, $f$ is isomorphic over the smooth locus of $X$. So
    if further $X$ is smooth in codimension $2$, then $T$ is isomorphic to $S$ by the construction, which gives the equality. 
\end{proof}

\begin{lem}\label{lem.it=ti}
    In Corollary~\ref{cor.rr.sA}, suppose that $X$ is Gorenstein along all crepant centers, then for any integer $s$ and any orbifold point in $B_X$, the local index satisfies $i_{f^{\lfloor *\rfloor}(sA)}=s\cdot i_{f^{\lfloor *\rfloor}(A)}.$ 
\end{lem}
\begin{proof}
    Let $f\colon Y\to X$ be a sequential terminalization, then $f$ is isomorphic in codimension $1$ outside crepant centers of $X$. Hence by assumption, locally at a non-Gorenstein point $P\in Y$ there is no exceptional divisor, which means that locally we have 
    \[
        f^{\lfloor *\rfloor}(sA)=f^*(sA)=sf^*(A)=sf^{\lfloor *\rfloor}(A).
    \]
    This concludes the lemma.
\end{proof}

\section{Geometry of foliations of rank two}\label{sec.foliation}

We gather some basic notions and facts regarding foliations on varieties. We refer the reader to \cite{druel}*{\S\,3} and the references therein for a more detailed explanation. 

\begin{defn}
    A {\it foliation} on a normal variety $X$ is a non-zero coherent subsheaf $\mathcal{F}$ of the tangent sheaf $\mathcal{T}_X$ such that 
    \begin{enumerate}
	\item $\mathcal{T}_X/\mathcal{F}$ is torsion-free, and	
	\item $\mathcal{F}$ is closed under the Lie bracket.
    \end{enumerate} 
    The {\it canonical divisor} of a foliation $\mathcal{F}$ is any Weil divisor $K_{\mathcal{F}}$ on $X$ such that $\det(\mathcal{F})\cong \mathcal{O}_X(-K_{\mathcal{F}})$. 
 
    Let $X_{\circ}\subset X_{\reg}$ be the largest open subset over which $\mathcal{T}_X/\mathcal{F}$ is locally free. A \emph{leaf} of $\mathcal{F}$ is a maximal connected and immersed holomorphic submanifold $L\subset X_{\circ}$ such that $\mathcal{T}_L=\mathcal{F}|_L$. A leaf is called \emph{algebraic} if it is open in its Zariski closure and a foliation $\mathcal{F}$ is said to be \emph{algebraically integrable} if its leaves are algebraic. 
\end{defn}

Let $\mathcal{F}$ be an algebraically integrable foliation on a normal projective variety $X$. Then there exists a diagram, called the \emph{family of leaves}, as follows (\cite{druel}*{\S\,3.6}):
\begin{equation}\label{eq.familyofleaves}
    \begin{tikzcd}[row sep=large, column sep=large]
	U \arrow[r,"e"] \arrow[d,"g"]
	& X \\
	T
    \end{tikzcd}
\end{equation}
with the following properties: 
\begin{enumerate}
    \item $U$ and $T$ are normal projective varieties;
    \item $e$ is birational and $g$ is an equidimensional fibration;
    \item For a general $t\in T$, $e$ is finite on $g^{-1}(t)$ and the image $e(g^{-1}(t))$ is the closure of a leaf of $\mathcal{F}$. 
\end{enumerate} 
Assume in addition that $K_{\mathcal{F}}$ is $\mathbb{Q}$-Cartier, there exists a canonically defined effective $e$-exceptional $\mathbb{Q}$-divisor $\Delta_U$ such that 
\begin{equation}\label{eq.PullBackKF}
    K_U-g^*K_T-R(g) + \Delta_U\sim K_{e^{-1}\mathcal{F}} + \Delta_U \sim_{\mathbb{Q}} e^*K_{\mathcal{F}},
\end{equation}
where 
\[
    R(g)\coloneq\sum_P(g^*P-(g^*P)_{\red})
\]
is the \emph{ramification divisor} of $g$, where $P$ runs through all prime divisors on $T$, see \cite{druel}*{\S\,3.6}. 

In this paper, we are mainly interested in a foliation of rank $2$ on a $\mathbb Q$-factorial Fano variety of Picard number $1$. Unlike the rank $1$ case (cf. \cite{liu-liu}*{Proposition 3.8} and \cite{jiang-liu-liu}*{Proposition 3.6}) where a general fiber of $g$ in \eqref{eq.familyofleaves} is a smooth rational curve, the study of $-K_{\mathcal F}$ in the rank $2$ case is more challenging as a general fiber of $g$ is not necessarily smooth. So we use the following lemma to pass to a modification of $U$.

\begin{lem}\label{lem.termoffoliation}
    Let $\mathcal{F}$ be an algebraically integrable foliation on a normal projective variety $X$ such that $K_{\mathcal{F}}$ is $\mathbb{Q}$-Cartier. Then there is a commutative diagram
    \begin{equation}\label{eq.termoffolation}
        \begin{tikzcd}[row sep=large, column sep=large]
		V \arrow[rr,"\mu", bend left=20]\arrow[r,"\pi"] \arrow[dr,"f"] & U \arrow[r,"e"] \arrow[d,"g"] & X \\
		  & T
	\end{tikzcd}
    \end{equation} 
    with the following properties:
    \begin{enumerate}
        \item $U, T, e, g$ are as in \eqref{eq.familyofleaves};
        \item $\pi$ is projective birational;
        \item $V$ is $\mathbb{Q}$-factorial terminal and $K_V$ is nef over $U$;
        \item There exists a canonically defined $\mathbb{Q}$-divisor $\Delta_V$ on $V$ such that 
            \begin{equation}\label{eq.PullBackKFtoV}
	       K_{\mu^{-1}\mathcal{F}} + \Delta_V \sim_{\mathbb{Q}} \pi^*(K_{e^{-1}\mathcal{F}} + \Delta_U )\sim_{\mathbb{Q}}\mu^*K_{\mathcal{F}},
            \end{equation} 
            where the $f$-horizontal part of $\Delta_V$ is effective. 
    \end{enumerate}
\end{lem}

\begin{proof}
    The existence of $V$ is by \cite{BCHM}*{Theorem~1.2}, namely, $V$ is a minimal model of a resolution of $U$ over $U$. Then $\Delta_V$ is canonically defined by \eqref{eq.PullBackKFtoV}. 

    In order to show that the $f$-horizontal part of $\Delta_V$ is effective, we may shrink $T$ so that $K_{e^{-1}\mathcal{F}}\sim K_U$ and $K_{\mu^{-1}\mathcal{F}}\sim K_V$ by \eqref{eq.PullBackKF}. So \eqref{eq.PullBackKFtoV} shows that $-\Delta_V$ is nef over $U$ and $\pi_*\Delta_V=\Delta_U$ is effective, which implies that $\Delta_V$ is effective by the negativity lemma (\cite{kollar-mori}*{Lemma~3.39}). 
\end{proof}

\begin{thm}\label{thm.hirzebruch}
    Let $X$ be a $\mathbb Q$-factorial canonical Fano variety of dimension $d\geq 3$ and of Picard number $1$. Let $\mathcal{F}$ be a foliation of rank $2$ on $X$ such that $\mu_{c_1(X),\min}(\mathcal F)>0$. Let $F$ be a general fiber of $f$ in \eqref{eq.termoffolation} and $G\coloneq \pi(F)$ a general fiber of $g$. If $-K_{\mathcal F}+\frac{5}{6}K_X$ is ample, then there exists a birational morphism $h\colon F\to \mathbb{F}_n$ for some $n\in\{1,2\}$. Furthermore, if $n=2$, then $h$ is isomorphic over $\sigma_0$ and $h_*^{-1}\sigma_0$ is exceptional over $G$. Here $\mathbb F_n$ is the $n$-th Hirzebruch surface and $\sigma_0$ is the negative section of $\mathbb F_n$.
\end{thm}

\begin{proof}
    As $X$ is of Picard number $1$, we may write $-K_{\mathcal{F}}\equiv -uK_X$ where $u>\frac{5}{6}$ by assumption. 

    As $\mu_{c_1(X),\min}(\mathcal F)>0$, by \cite{cp}*{Theorem 1.1} or \cite{ouwenhao}*{Proposition 2.2}, $\mathcal F$ is algebraically integrable and a general fiber $G$ of $g$ in \eqref{eq.familyofleaves} is rationally connected. As $V$ has only terminal singularities, $F$ is a smooth rationally connected surface. By Lemma~\ref{lem.termoffoliation}, there exists an effective $\mathbb{Q}$-divisor $\Delta_F\coloneq \Delta_V|_F$ such that 
    \[
        -K_F-\Delta_F\equiv \mu^*(-K_{\mathcal F})|_F\equiv \mu^*(-uK_X)|_F
    \]
    is nef and big. On the other hand, as $X$ has only canonical singularities, there exists an effective $\mathbb{Q}$-divisor $\Delta'$ on $V$ such that $ K_V=\mu^*(K_X)+\Delta'$. Hence there exists an effective $\mathbb{Q}$-divisor $\Delta'_F\coloneq \Delta'|_F$ such that 
    \[
        -K_F+\Delta'_F\equiv\mu^*(-K_X)|_F.
    \]

    By a standard minimal model program, there exists a birational morphism $F\to S$ where $S$ is either $\mathbb P^2$ or a Hirzebruch surface $\mathbb F_n$ for some $n\geq 0$.

    First, we claim that $F$ is not $\mathbb{P}^2$ or $\mathbb{F}_0$. Suppose that $F$ is either $\mathbb{P}^2$ or $\mathbb{F}_0$, then $\pi\colon F\to G$ is the identity morphism as $F$ admits no non-trivial contraction. Hence the coefficients of $\Delta_F$ are at least $1$ by \cite{jiang-liu-liu}*{Proposition 3.5}. If $F=\mathbb{P}^2$, then for a line $\ell$ on $F$, $\Delta_F\cdot \ell\geq 1$ and hence
    \[ 
        \frac{5}{6}<u=\frac{\mu^*(-uK_X)|_F\cdot \ell }{\mu^*(-K_X)|_F \cdot \ell }=\frac{(-K_F-\Delta_F)\cdot \ell }{(-K_F+\Delta'_F)\cdot \ell }\leq \frac{2}{3},
    \] 
    which is a contradiction.
    If $F=\mathbb{F}_0=\mathbb{P}^1\times \mathbb{P}^1$, then there exists a ruling $\ell$ on $F$ such that $\Delta_F\cdot \ell\geq 1$ and hence
    \[ 
        \frac{5}{6}<u=\frac{\mu^*(-uK_X)|_F\cdot \ell }{\mu^*(-K_X)|_F \cdot \ell }=\frac{(-K_F-\Delta_F)\cdot \ell }{(-K_F+\Delta'_F)\cdot \ell }\leq \frac{1}{2},
    \]
    which is a contradiction.

    So if $S$ is either $\mathbb{P}^2$ or $\mathbb{F}_0$, then $F\to S$ is non-trivial and factors through a blowup of $S$, which dominates $\mathbb{F}_1$. Hence we conclude that there exists a birational morphism $h\colon F\to \mathbb F_n$ for some $n\geq 1$.

    Then we show that $n\leq 2$. Let $\ell$ be a general fiber of the natural map $ \mathbb{F}_n\to \mathbb{P}^1$, and let $\sigma_0', \ell'$ be the strict transform of $\sigma_0, \ell$ on $F$. Then $\sigma_0'^2\leq \sigma_0^2=-n$ and $(\sigma_0'\cdot \ell')=1$. 
 
    Set $a\coloneq \mult_{\sigma_0'}(\Delta_F)\geq 0$. It follows that 
    \[
        \frac{5}{6}< u=\frac{\mu^*(-uK_X)|_F\cdot \ell'}{\mu^*(-K_X)|_F \cdot \ell'}=\frac{(-K_F-\Delta_F)\cdot \ell'}{(-K_F+\Delta'_F)\cdot \ell'}\leq \frac{2-a}{2},
    \]
    so $a<\frac13$. On the other hand, as $-K_F-\Delta_F$ is nef, we have
    \begin{align}\label{eq.n<=2}
        0\leq (-K_F-\Delta_F)\cdot \sigma_0'\leq (-K_F-a \sigma_0')\cdot \sigma_0' =2+(1-a)\sigma_0'^2< 2-\frac{2n}{3},
    \end{align}
    which implies that $n\leq 2$.

    From now on, suppose that $n=2$. Then \eqref{eq.n<=2} implies that $\sigma_0'^2=-2$, which implies that $h$ is isomorphic over $\sigma_0$. If $a>0$, then $\sigma_0'$ is exceptional over $G$ as the coefficients of $\Delta_G$ are at least $1$ by \cite{jiang-liu-liu}*{Proposition 3.5}, where $\Delta_G\coloneq \Delta_U|_G$. If $a=0$, then 
    \begin{align*}
        0\leq (-K_F-\Delta_F)\cdot \sigma_0'=-\Delta_F\cdot \sigma_0'\leq 0,
    \end{align*}
    which implies that $(-K_F-\Delta_F)\cdot \sigma_0'=0$, then $\sigma_0'$ is exceptional over $G$ as $-K_F-\Delta_F=\pi^*(-K_G-\Delta_G)$ and $-K_G-\Delta_G=e^*(-K_{\mathcal{F}})|_G$ is ample. Here we used the fact that $G\to X$ is finite by the construction of \eqref{eq.familyofleaves} (\cite{druel}*{\S\,3.6}).
\end{proof}

\section{Ruling out the exceptional cases}

Now we are prepared to rule out the two remaining cases in Theorem~\ref{thm.isolatecase}.

\begin{thm}\label{thm.rulingout}
    Let $X$ be a $\mathbb Q$-factorial Fano $3$-fold of Picard number $1$ with at worst isolated canonical singularities. Then it can not happen that $\qQ(X)\in\{67, 71\}$ with the numerical data of $X$ listed in Table~\ref{tab1}.
\end{thm}

\begin{proof}
    Suppose that $X$ is a $\mathbb Q$-factorial Fano $3$-fold of Picard number $1$ with at worst isolated canonical singularities such that $\qQ(X)\in\{67, 71\}$ with the numerical data of $X$ listed in Table~\ref{tab1}. For simplicity, denote $q\coloneq \qQ(X)$. 

First, we show that $\Cl(X)$ has no non-trivial torsion, so $\Cl(X)\simeq \mathbb{Z}$.
Indeed, suppose that there exists a Weil divisor $D$ such that $D\equiv 0$ but $D\not\sim 0$, then by Lemma~\ref{lem.torsionRR},
\begin{align}
        2=\sum_{(r,b)\in B_X}\frac{\overline{ i_{f^{\lfloor *\rfloor}(D)} b}(r-\overline{i_{f^{\lfloor *\rfloor}(D)}b})}{2r},\label{eq X no torsion}
    \end{align}
    where $i_{f^{\lfloor *\rfloor}(D)} $ is the local index of ${f^{\lfloor *\rfloor}(D)}$ at the orbifold point of type $(r,b)\in B_X$. Here we used the fact that $(f^{\lfloor *\rfloor}(D))^2\cdot K_Y=D^2\cdot K_X=0$ by Lemma~\ref{lem.DH}.
    But by comparing the denominators, \eqref{eq X no torsion} is absurd according to the $B_X$ in Table~\ref{tab1}. 
    
    Let $A$ be an ample generator of $\Cl(X)$. For any Weil divisor $D$ on $X$, we define $\iota(D)\in \mathbb{Z}$ such that $D\sim \iota(D) A$. In particular, $\iota(-K_X)=q$. 

    \begin{claim}\label{claim.rrfor<=33}
        \begin{equation*}
            h^0(X, \mathcal O_X(sA))=
            \begin{dcases}
            0 & \text{if } s=1,2,3,4,7,8,9,13,14; \\ 
            1 & \text{if } s=5,6,10,11,12,15,16; \\ 
            3& \text{if } s=33.
            \end{dcases}
        \end{equation*}
    \end{claim}
    \begin{proof}
        We apply Corollary~\ref{cor.rr.sA} to compute $h^0(X, \mathcal O_X(sA))$. Let $f\colon Y\to X$ be a sequential terminalization. Then 
        \[
            -f^{\lfloor *\rfloor}(sA))^2\cdot K_Y=-(sA)^2\cdot K_X=-\frac{s^2}{q^2}K_X^3=\frac{s^2}{330}
        \] 
        by Lemma~\ref{lem.DH} and Table~\ref{tab1}. By Theorem \ref{thm.isolatecase} or \cite{kawakita}*{Table 2}, $X$ is Gorenstein along all crepant centers. Hence the local index $i_s=si_1$ at any orbifold point in $B_X$ by Lemma~\ref{lem.it=ti}. So we only need to determine $i_1$ (or equivalently $\overline{i_1b}$) for each $(r,b)\in B_X$. By the fact that $h^0(X, \mathcal{O}_X(A))$ is an integer, we can get all possible values of the corresponding $\overline{i_1b}$ as follows:
        \[
            \overline{i_1b}=\begin{cases}
            1 & \text{if } r=2;\\
            1 \text{ or } 2 & \text{if } r=3;\\
            2 \text{ or } 3 & \text{if } r=5;\\
            2 \text{ or } 9 & \text{if } r=11.
            \end{cases}
        \]
        So we conclude the lemma by Corollary~\ref{cor.rr.sA}.
    \end{proof}

    \begin{claim}\label{claim.easyfact}
        If $D\in |sA|$ is non-reduced for some positive integer $s\leq 33$, then $D\geq 2A_5$ or $D\geq 2A_6$, where $A_5\in|5A|$ and $A_6\in|6A|$ are the unique elements.
    \end{claim}

    \begin{proof}
        Let $P$ be a prime divisor on $X$ such that $D\geq 2P$, then $\iota(P) \leq \frac{\iota(D)}{2}<17$. By Claim~\ref{claim.rrfor<=33}, we have $P\in |5A|$ or $|6A|$. 
    \end{proof}

    \begin{claim}\label{claim.prepare}
        There exists a foliation $\mathcal F$ of rank $2$ on $X$ such that $\mu_{c_1(X), \min}(\mathcal F)>0$ and $\iota(-K_{\mathcal F})\geq \max\{ q-10, \frac{57}{67}q\}$. 
    \end{claim}

    \begin{proof}
        By Theorem~\ref{thm.isolatecase}, $(l, r_1)=(2,2)$ or $(3,1)$ in both cases. So we can take $\mathcal F=\mathcal E_{l-1}$ in the Harder--Narasimhan filtration of $\mathcal{T}_X$ \eqref{eq.hnfiltration}, where $\mu_{c_1, \min}(\mathcal F)>\mu_{c_1}(\mathcal T_X/\mathcal F)\geq 0$ by definition and \cite{ouwenhao}*{Theorem 1.4}. Hence $\mathcal F$ is a foliation of rank $2$ on $X$ (see \cite{cp}*{Theorem 1.1} or \cite{ouwenhao}*{Proposition 2.2}). By Theorem~\ref{thm.kmineqforisolated} where $p=\iota(-K_\mathcal{F})$, we have 
        \[
            \frac{q^2}{1361}=\frac{c_1(X)^3}{c_2(X)\cdot c_1(X)}\leq \max\left\{\frac{4q^2}{p(4q-3p)} , \frac{4q^2}{-4p^2+6pq-q^2}\right\}=\frac{4q^2}{-4p^2+6pq-q^2}.
        \]
        Note that $p>\frac{2}{3}q$. From the above inequality, we conclude that $p\geq 68$ or $p\geq 57$ when $q=71$ or $67$ respectively.
    \end{proof}

    From now on, we use the notation in Lemma~\ref{lem.termoffoliation} and Theorem~\ref{thm.hirzebruch}. Recall that as  \eqref{eq.termoffolation} we have the following commutative diagram
    \[
        \begin{tikzcd}[row sep=large, column sep=large]
		V \arrow[rr,"\mu", bend left=20]\arrow[r,"\pi"] \arrow[dr,"f"] & U \arrow[r,"e"] \arrow[d,"g"] & X \\
		  & T
	\end{tikzcd}
    \]
    where $F$ is a general fiber of $f$ and $G\coloneq \pi(F)$ is a general fiber of $g$. Denote $p=\iota(-K_\mathcal{F})$. 

    \begin{claim}\label{claim.mu(F)}
        $\iota(\mu_*F)>33$, that is, $\mu_*F-33A$ is ample. 
    \end{claim}

    \begin{proof}
        Note that $\mu_*F=e_*G$. 
        
        As $X$ is rationally connected, $V$ is rationally connected. Hence $T\cong\mathbb{P}^1$ and $-g^*K_T\sim 2G$. By \eqref{eq.PullBackKF}, we have
        \begin{equation}\label{eq.eR}
            e_*R(g)\equiv e_*K_U+e_*(2G)-K_{\mathcal F}=2e_*G-(q-p)A,
        \end{equation}
        where $R(g)$ is the ramification divisor. As $|e_*G|$ is a movable linear system, by Claim~\ref{claim.rrfor<=33}, $\iota(e_*G)>16$.  So by Claim~\ref{claim.prepare}, 
        \begin{align}
            \iota(e_*R(g))=2\iota(e_*G)-q+p>\iota(e_*G). \label{eq:R(g)>G}
        \end{align}
        By the definition of $R(g)$, there are non-reduced divisors $G_1,\dots, G_k\in |e_*G|$ such that
        \[
            e_*R(g)=\sum_{i=1}^k(G_i-(G_i)_{\text{\rm red}}).
        \]
        This implies that $\iota(e_*R_(g))<k\iota(e_*G)$, and hence $k>1$ by \eqref{eq:R(g)>G}. Moreover, pairwisely $G_1, \dots, G_k$ have no common components as they come from pushforwards of different fibers of $g$. If $\iota(e_*G)\leq 33$, then by Claim~\ref{claim.easyfact}, each $G_i$ contains either $2A_5$ or $2A_6$, which implies that $k\leq 2$. So we have $k=2$ and $(G_1)_{\text{\rm red}}+(G_2)_{\text{\rm red}}\geq A_5+A_6$. Then by Claim~\ref{claim.prepare},
        \[
            \iota(e_*R(g))\leq 2\iota(e_*G)-5-6<2\iota(e_*G)-q+p,
        \]
        which contradicts \eqref{eq:R(g)>G}.
    \end{proof}

    Write $\mu^*|sA|=|M_{s}|+N_{s}$ for $s\in \{5,6,33\}$, where $N_s$ is the fixed part and $|M_{s}|$ is the movable part with a general element $M_s$. Here we emphasize that $M_s$ is an effective Weil divisor and $N_s$ is an effective $\mathbb{Q}$-divisor whose fractional part is $\mu$-exceptional. By Lemma~\ref{lem.termoffoliation}, there exists an effective $\mathbb{Q}$-divisor $\Delta_F\coloneq \Delta_V|_F$ such that 
    \begin{equation}\label{eq.linearity}
        -K_F-\Delta_F\equiv \mu^*(-K_\mathcal{F})|_F\equiv\frac{p}{s}\mu^*(sA)|_F\equiv\frac{p}{s}(M_s+N_s)|_F
    \end{equation}
    is nef and big. In particular,
    \[
        -K_F-\Delta_F \equiv\frac{p}{5}N_{5}|_F\equiv \frac{p}{6}N_{6}|_F.
    \]
    Here $M_5=M_6=0$ as $ h^0(X, \mathcal O_X(5A))=h^0(X, \mathcal O_X(6A))=1$.

    By Theorem~\ref{thm.hirzebruch} and  $p\geq \frac{57}{67}q$ from  Claim~\ref{claim.prepare}, there exists a morphism $h\colon F\to \mathbb F_n$ for some $n\in \{1,2\}$. Then $h_*(N_{5}|_F)$ and $h_*(N_{6}|_F)$ are nef and big $\mathbb{Q}$-divisors on $\mathbb{F}_n$. We may write 
    \begin{align*}
        h_*(N_{5}|_F)\equiv a_5\sigma_0+b_5\ell,\\
        h_*(N_{6}|_F)\equiv a_6\sigma_0+b_6\ell,
    \end{align*}
    for some positive rational numbers $a_5,b_5, a_6, b_6$, where $\sigma_0$ and $\ell$ are the negative section and the ruling of $\mathbb{F}_n$ respectively. 

    \begin{claim}\label{claim.final}
        \begin{enumerate}
            \item $\frac{a_5}{a_6}=\frac{b_5}{b_6}=\frac56$.
            \item $a_5\leq \frac{10}{p}$,  $a_6\leq \frac{12}{p}$.
            \item If $n=1$, then $b_5\leq \frac{15}{p}$.
            \item If $n=2$, then $b_5=2a_5$.
            \item $a_5=\frac{5}{q}$ or $\frac{10}{q}$.
        \end{enumerate}
    \end{claim}
    
    \begin{proof}
        (1) This follows directly from $6N_5|_F\equiv 5N_6|_F$. 

        (2) For $s=5,6$, we have
        \begin{equation}
            h_*(N_{s}|_F)\equiv \frac{s}{p}h_*(-K_F-\Delta_F)\leq \frac{s}{p}h_*(-K_F)= -\frac{s}{p}K_{\mathbb{F}_n}.\label{eq N<-K}
        \end{equation} 
        Intersecting with $\ell$, we get $a_s\leq \frac{2s}{p}$.

        (3)  If $n=1$, intersecting \eqref{eq N<-K} with the nef divisor $\sigma_0+\ell$, we have $b_5\leq \frac{15}{p}$. 

        (4) By Theorem~\ref{thm.hirzebruch}, $h$ is isomorphic over $\sigma_0$ and $h_*^{-1}\sigma_0$ is exceptional over $G$, so $h_*(-K_F-\Delta_F)\cdot \sigma_0=0$. This means that $h_*(N_{5}|_F)\cdot \sigma_0=0$ and $b_5=2a_5$. 
 
        (5) Write $N_5=\sum_{i}c_iB_i$ and $N_6=\sum_{i}c'_iB_i$ where $B_i$ are prime divisors. We claim that $c_i\in \{\frac{5}{q}, \frac{10}{q}\}$ as long as $h_*(B_i|_F)\cdot \ell>0$. If this is true, then 
        \[
            a_5=h_*(N_5|_F)\cdot \ell=\sum_ic_ih_*(B_i|_F)\cdot \ell\in \frac{5}{q}\mathbb{Z},
        \]
        which implies that $a_5\in \{\frac{5}{q}, \frac{10}{q}\}$ as $a_5\leq \frac{10}{p}$.

        Note that $N_5=\mu^*(5A)$ and $N_6=\mu^*(6A)$, so we have $6N_5-5N_6\sim 0$, in particular, $6c_i-5c_i'\in \mathbb{Z}$ for each $i$. 

        Take any $B_i$ with $h_*(B_i|_F)\cdot \ell>0$. For simplicity we just denote $B=B_i$, $c=c_i$ and $c'=c'_i$. Then $c\leq h_*(N_5|_F)\cdot \ell=a_5\leq \frac{10}{p}<1$ and similarly $c'\leq \frac{12}{p}$, which implies that $B$ is $\mu$-exceptional and is contracted to a singular point $P$ on $X$. 

        If $P$ is a terminal point on $X$, then $rA$ is Cartier at $P$ for some $r\in \{1,2,3,5,11\}$ by \cite{kawamata-crepant}*{Corollary~5.2}, which implies that $rc, rc'$ are positive integers. As $c\leq \frac{10}{p}<\frac15$ and $c'\leq \frac{12}{p}$, we have $r=11$, $c=\frac{1}{11}$, and $c'\in \{\frac{1}{11}, \frac{2}{11}\}.$ But this contradicts that $6c-5c'\in\mathbb{Z}$. 

        If $P$ is a non-terminal point (i.e., a crepant center) on $X$, then $qA\sim -K_X$ is Cartier at $P$ by Theorem~\ref{thm.isolatecase} or \cite{kawakita}*{Table 2}, which implies that $qc, qc'$ are positive integers. As $c\leq \frac{10}{p}<\frac{12}{q}$ and $c'\leq \frac{12}{p}<\frac{15}{q}$, we have $c\in\{\frac{1}{q}, \dots, \frac{11}{q}\}$, $c'\in\{\frac{1}{q}, \dots, \frac{14}{q}\}$. As $q$ is prime and $6c-5c'\in\mathbb{Z}$, we have $c\in\{\frac{5}{q}, \frac{10}{q}\}$. This concludes the claim.
    \end{proof}
 
    Finally, by 
    \[
        M_{33}+N_{33}\equiv \mu^*(33A)\equiv  3(N_5+N_6)
    \] 
    and Claim~\ref{claim.final}, we have
    \begin{align}
        h_*(M_{33}|_F)+h_*(N_{33}|_F)\equiv 3(a_5+a_6)\sigma_0+3(b_5+b_6)\ell \leq \frac{66}{q}(\sigma_0+2\ell). \label{eq.final}
    \end{align}
    As $h^0(X, \mathcal{O}_X(33A))=3$, $M_{33}|_F$ is a non-trivial nef Weil divisor, so $h_*(M_{33}|_F)$ is nef and can be written as $h_*(M_{33}|_F)\sim a\sigma_0+b\ell$ for some integers $b\geq na\geq 0$.  However, \eqref{eq.final} implies that $a<1$ and $b<2$, hence $h_*(M_{33}|_F)\sim \ell$ and $h^0(\mathbb F_n, \mathcal O_{\mathbb F_n}(h_*(M_{33}|_F)))=2$. On the other hand, as $33A-\mu_*F$ is anti-ample by Claim~\ref{claim.mu(F)}, we have
    \[
        h^0(V, \mathcal O_V(M_{33}-F))\leq h^0(X, \mathcal O_X(33A-\mu_*F))=0,
    \]
    and hence
    \begin{align*}
        h^0(\mathbb F_n, \mathcal O_{\mathbb F_n}(h_*(M_{33}|_F)))\geq {}& h^0(F, \mathcal O_F(M_{33}|_F))\\
        \geq {}& h^0(V, \mathcal O_V(M_{33}))- h^0(V, \mathcal O_V(M_{33}-F))\\
        = {}& h^0(V, \mathcal O_V(M_{33}))  = 3, 
    \end{align*}  
    which is absurd.
\end{proof}

\begin{proof}[Proof of Theorem~\ref{thm.main}]
    By Lemma~\ref{lem.reducetopicard1}, we may assume that $X$ is of Picard number $1$, then the result follows from Theorem~\ref{thm.isolatecase} and Theorem~\ref{thm.rulingout}.
\end{proof}

\section*{Acknowledgments} 
The authors would like to thank Jie Liu for helpful discussions. C.~Jiang was supported by National Key Research and Development Program of China (No. 2023YFA1010600, No. 2020YFA0713200) and NSFC for Innovative Research Groups (No. 12121001). C.~Jiang is a member of the Key Laboratory of Mathematics for Nonlinear Sciences, Fudan University. H.~Liu is supported in part by the National Key Research and Development Program of China (No. 2023YFA1009801).

\end{document}